\documentclass[a4paper,12pt]{amsart}
\usepackage[margin=2.8cm]{geometry}
\usepackage[applemac]{inputenc}
\usepackage{eurosym}
\newtheorem{theorem}{Theorem}[section]
\newtheorem{rem} [theorem] {Remark}

\newtheorem{lemma}[theorem]{Lemma}

\newcommand{\ovprt}{\overline{\partial}}
\newcommand{\ovli}{\overline}
\newcommand{\dquer}{\overline\partial}
\newcommand{\dquers}{\overline\partial ^*_\varphi}
\newcommand{\boxphi}{\square_\varphi}
\newcommand{\levim}{\frac{\partial^2\varphi}{\partial z_j\partial\overline z_k}}

\numberwithin{equation}{section}

\title{ Spectrum of the $\ovprt$-Neumann Laplacian on the Fock space.}

\author{ Friedrich Haslinger}

\thanks{Partially supported by the FWF-grant  P23664.}

 \address{ F. Haslinger: Institut f\"ur Mathematik, Universit\"at Wien,
Nordbergstrasse 15, A-1090 Wien, Austria}
\email{ friedrich.haslinger@univie.ac.at}


\begin{document}

\maketitle

\begin{abstract} ~\ The spectrum of the $\ovprt$-Neumann Laplacian on the Fock space
$L^2(\mathbb{C}^n, e^{-|z|^2})$ is explicitly computed. It turns out that it consists of positive integer eigenvalues each of which is of infinite multiplicity. Spectral analysis of the $\ovprt$-Neumann Laplacian on the Fock space is closely related to Schr\"odinger operators with magnetic field and to the complex Witten-Laplacian.
\end{abstract}

\section{Introduction.}~\\

The spectrum of the $\ovprt$-Neumann Laplacian for the ball and annulus was  explicitly computed by Folland (\cite{Fo72}), Siqi Fu (\cite{Fu1}) determined the spectrum for the polydisc showing that it needs not being purely discrete like for the usual Dirichlet Laplacian. Here we will exhibit the weighted case, where the weight function is $\varphi (z)=|z|^2,$ showing that the essential spectrum is nonempty, which is equivalent to the fact that the $\ovprt$-Neumann operator (the inverse to the $\ovprt$-Neumann Laplacian) fails to be compact (\cite{Dav}).

\vskip 0.3 cm

Let $\varphi : \mathbb C^n \longrightarrow \mathbb R^+ $ be a plurisubharmonic $\mathcal C^2$-weight function and define the space
$$L^2(\mathbb{C}^n, e^{-\varphi})=\{ f:\mathbb C^n \longrightarrow \mathbb C \ : \ \int_{\mathbb C^n}
|f|^2\, e^{-\varphi}\,d\lambda < \infty \},$$
where $\lambda$ denotes the Lebesgue measure, the space $L^2_{(0,q)}(\mathbb C^n, e^{\varphi} )$ of $(0,q)$-forms with coefficients in
$L^2(\mathbb{C}^n, e^{-\varphi}),$ for $1\le q \le n.$ 
Let 
$$( f,g)_\varphi=\int_{\mathbb{C}^n}f \,\overline{g} e^{-\varphi}\,d\lambda$$
denote the inner product and 
$$\| f\|^2_\varphi =\int_{\mathbb{C}^n}|f|^2e^{-\varphi}\,d\lambda $$
the norm in $L^2(\mathbb{C}^n, e^{-\varphi}).$

We consider the weighted
$\ovprt $-complex 
$$
L^2_{(0,q-1)}(\mathbb C^n , e^{-\varphi} )\underset{\underset{\ovprt_\varphi^* }
\longleftarrow}{\overset{\ovprt }
{\longrightarrow}} L^2_{(0,q)}(\mathbb C^n , e^{-\varphi} )
\underset{\underset{\ovprt_\varphi^* }
\longleftarrow}{\overset{\ovprt }
{\longrightarrow}} L^2_{(0,q+1)}(\mathbb C^n , e^{-\varphi} ),
$$
where 
 for $(0,q)$-forms $u=\sum_{|J|=q}' u_J\,d\overline z_J$ with coefficients in $\mathcal{C}_0^\infty (\mathbb{C}^n)$ we have
$$\ovprt u = \sum_{|J|=q}\kern-1pt{}^{\prime}  \, \sum_{j=1}^n \frac{\partial u_J}{\partial \ovli z_j}\, d\ovli z_j \wedge d \ovli z_J,$$
and 
$$ \ovprt_\varphi ^* u = -  \sum_{|K|=q-1}\kern-7pt{}^{\prime}  \, \sum_{k=1}^n \delta_k u_{kK}\,d\ovli z_K,$$
where $\delta_k=\frac{\partial}{\partial z_k}-\frac{\partial \varphi}{\partial z_k}.$

There is an interesting  connection between $\ovprt $ and the theory of Schr\"odinger operators\index{Schr\"odinger operator} with magnetic fields,
see for example \cite{FS3}  for recent contributions exploiting this point of view.

The complex Laplacian on $(0,q)$-forms is defined as
$$\square_{\varphi,q} := \dquer  \,\dquers + \dquers \dquer,$$
where the symbol $\square_{\varphi,q} $ is to be understood as the maximal closure of the operator initially defined on forms with coefficients in $\mathcal{C}_0^\infty$, i.e., the space of smooth functions with compact support.

$\square_{\varphi,q} $ is a selfadjoint and positive operator, which means that 
$$(\square_{\varphi,q} f,f)_\varphi \ge 0 \ , \   {\text{for}} \  f\in dom (\boxphi ).$$
The associated Dirichlet form is denoted by 
\begin{equation}\label{diri}
Q_\varphi (f,g)= ( \dquer f,\dquer g)_\varphi + ( \dquers f ,\dquers g)_\varphi, 
\end{equation}
for $f,g\in dom (\dquer ) \cap dom (\dquers ).$ The weighted $\dquer $-Neumann operator
$N_{\varphi, q} $ is - if it exists - the bounded inverse of $\square_{\varphi,q} .$ 
\vskip 0.5 cm
We indicate that a square integrable $(0,1)$-form $f=\sum_{j=1}^n f_j\,d\ovli z_j$ belongs to  $dom(\dquers)$ if and only if 
\begin{equation*}
 e^{\varphi }\sum_{j=1}^n\frac{\partial}{\partial z_j}\left( f_je^{-\varphi}\right) \in L^2(\mathbb{C}^n, e^{-\varphi} ),
\end{equation*}
where the derivative is to be taken in the sense of distributions,
 and that forms with coefficients in $\mathcal{C}_0^\infty(\mathbb{C}^n)$ are dense in $dom(\dquer)\cap dom(\dquers)$ in the graph norm $f\mapsto (\Vert \dquer f\Vert _\varphi^2+\Vert \dquers f\Vert _\varphi^2)^\frac{1}{2}$ (see \cite{GaHa}).
\vskip 0.5 cm
We consider  the Levi - matrix 
$$M_\varphi=\left(\levim\right)_{jk}$$
of $\varphi $ and suppose that
 the sum $s_q$ of any $q$ (equivalently: the smallest $q$) eigenvalues of $M_\varphi$ satisfies
\begin{equation}\label{sq}
\liminf_{|z|\to \infty}s_q(z) > 0.
\end{equation}
We show that (\ref{sq}) implies that there exists a continuous linear operator
$$N_{\varphi, q} : L^2_{(0,q)}( \mathbb{C}^n, e^{-\varphi}) \longrightarrow L^2_{(0,q)}( \mathbb{C}^n, e^{-\varphi}),$$
such that $\square_{\varphi,q} \circ N_{\varphi,q} u = u,$ for any $u\in L^2_{(0,q)}( \mathbb{C}^n, e^{-\varphi}).$
\vskip 0.3 cm
If we suppose that  the sum $s_q$ of any $q$ (equivalently: the smallest $q$) eigenvalues of $M_\varphi$ satisfies
\begin{equation}\label{sqcomp}
\lim_{|z|\to \infty}s_q(z) =\infty .
\end{equation}
Then the $\ovprt$-Neumann operator $N_{\varphi, q} : L^2_{(0,q)}( \mathbb{C}^n, e^{-\varphi})\longrightarrow L^2_{(0,q)}( \mathbb{C}^n, e^{-\varphi})$ is compact (see \cite{Has5}, \cite{Has6} for further details).

To find the canonical solution to $\ovprt f=u,$ where $u \in L^2_{(0,1)}( \mathbb{C}^n, e^{-\varphi}) $ is a given $(0,1)$-form such that $\ovprt u=0,$ one can take $f= \ovprt^*_\varphi \, N_{\varphi, 1} u$ and $f$ will also satisfy $f\perp {\text{ker}}\ovprt.$ For further results on the canonical solution operator to $\ovprt$ see \cite{Has01_2} and \cite{HaLa}.
\vskip 0.3 cm
If the weight function is $\varphi (z)=|z|^2,$ the corresponding Levi matrix $M_\varphi$ is the identity. The space $A^2(\mathbb{C}^n, e^{-|z|^2})$ of entire functions belonging to $L^2(\mathbb{C}^n, e^{-|z|^2})$ is the Fock space, which plays an important role in quantum mechanics. In this case
\begin{equation}\label{box0}
\square_{\varphi,0} \, u= \ovprt_\varphi ^* \,\ovprt u = -\frac{1}{4}\, \triangle u + \sum_{j=1}^n \overline z_j \, u_{ \overline z_j},
\end{equation}
where $u\in {\text{dom}} \, \square_{\varphi,0} \subseteq L^2(\mathbb{C}^n, e^{-|z|^2})$ and
\begin{equation}\label{boxn}
\square_{\varphi,n} \, u = \ovprt \, \ovprt_\varphi ^* u = -\frac{1}{4}\, \triangle u + \sum_{j=1}^n \overline z_j \, u_{ \overline z_j} +n \, u,
\end{equation}
where $u\in {\text{dom}}\, \square_{\varphi,n} \subseteq L^2_{(0,n)}(\mathbb{C}^n, e^{-|z|^2}).$

For $n=1$ there is a connection to Schr\"odinger operators with magnetic field (see \cite{AHS} for properties of its spectrum), and to Dirac and Pauli operators \cite{HaHe}: the operators
$$P_+= e^{-|z|^2 /2}\,  \ovprt \, \ovprt^*_\varphi \ e^{|z|^2 /2}, \ \ P_- =e^{-|z|^2 /2}\,  \ovprt^*_\varphi \, \ovprt \ e^{|z|^2 /2}$$
defined on $L^2(\mathbb C)$ are the Pauli operators, $P_+$ is also a Schr\"odinger operator with magnetic field and the square of the corresponding Dirac operator satisfies
$$ \mathcal D^2=\left(
\begin{array}{cc}
 P_-&0\\0&P_+
\end{array}\right).$$

For $n>1$ and $1\le q \le n-1$ the $\ovprt$-Neumann Laplacian $\square_{\varphi,q}$ acts diagonally (see \cite{has3}): for 
$$u=\sum_{|J|=q}\kern-1pt{}^{\prime} u_J\,d\overline z_J  \in {\text{dom}}\, \square_{\varphi,q} \subseteq L^2_{(0,q)}(\mathbb{C}^n, e^{-|z|^2})$$ 
we have
\begin{equation}\label{boxq}
\square_{\varphi,q} \, u = (\ovprt \, \ovprt_\varphi ^* + \ovprt_\varphi ^* \,\ovprt ) \, u = \sum_{|J|=q}\kern-1pt{}^{\prime} \left (
 -\frac{1}{4}\, \triangle u_J + \sum_{j=1}^n \overline z_j \, u_{J \overline z_j} +q \, u_J \right )\,d\overline z_J .
 \end{equation}

\section{Determination of the spectrum.}~\\

In order to determine the spectrum of $\square_{\varphi,q}$ for $\varphi (z)=|z|^2$ we use the following lemma (see \cite{Dav}, Lemma 1.2.2):

\begin{lemma}\label{davies}
Let $H$ be a symmetric operator on a Hilbert space $\mathcal{H}$ with domain $L$ and let $(f_k)_{k=1}^\infty$ be a complete orthonormal set in $\mathcal{H}.$ If each $f_k$ lies in $L$ and there exist $\mu_k \in \mathbb{R}$ such that $Hf_k=\mu_k f_k$ for every $k,$ then $H$ is essentially self-adjoint . Moreover, the spectrum of $\overline H$ is the closure in $\mathbb{R}$ of the set of all $\mu_k.$
\end{lemma}

\vskip 0.3 cm

In sake of simplicity and in order to explain the general method, we start with the complex one-dimensional case. Looking for the eigenvalues $\mu$ of  $\square_{\varphi,0} $ we have by \eqref{box0}:
\begin{equation}\label{spect1}
\square_{\varphi,0} \, u= -u_{z\, \overline z} + \overline z u_{\overline z} = \mu u
\end{equation}
It is clear that the space $A^2(\mathbb{C}^n, e^{-|z|^2})$ is contained in the eigenspace of the eigenvalue $\mu =0.$

For any positive integer $k$ the antiholomorphic monomial $\overline z^k$ is an eigenfunction for the eigenvalue $\mu =k.$ 

In the following we denote by $\mathbb{N}_0=\mathbb{N}\cup \{0\}.$

\begin{lemma}\label{eigenfun}
Let $n=1 .$  For $k\in \mathbb{N}_0$ and $m \in \mathbb{N}$  the functions
\begin{equation}\label{eig1}
u_{k,m}(z,\overline z)= \overline z^{k+m} \, z^m + \sum_{j=1}^m  \frac{(-1)^j (k+m)!\,  m!}{j! \, (k+m-j)!\, (m-j)!} \, \overline z^{k+m-j}\, z^{m-j}
\end{equation}
are eigenfunctions for the eigenvalue $k+m$ of the operator $\square_{\varphi,0} \, u= -u_{z\, \overline z} + \overline z u_{\overline z}.$

For $k\in \mathbb{N}$ and $m \in \mathbb{N}_0$ the functions
\begin{equation}\label{eig2}
v_{k,m}(z,\overline z)= \overline z^{k} \, z^{k+m} + \sum_{j=1}^k  \frac{(-1)^j (k+m)!\,  k!}{j! \, (k+m-j)!\, (k-j)!} \, \overline z^{k-j}\, z^{k+m-j}
\end{equation}
are eigenfunctions for the eigenvalue $k$ of the operator $\square_{\varphi,0} \, u= -u_{z\, \overline z} + \overline z u_{\overline z}.$
\end{lemma}

\begin{proof}
To prove \eqref{eig1} we set
$$u_{k,m}(z,\overline z)= \overline z^{k+m} \, z^m + a_1 \overline z^{k+m-1}\, z^{m-1} + a_2 \overline z^{k+m-2}\, z^{m-2} + \dots + a_{m-1} \overline z^{k+1} \, z + a_m \overline z^k
$$
and compute
\vskip 0.2 cm
$\frac{\partial^2}{\partial z \partial \overline z}\, u_{k,m}(z,\overline z)$
$$ = (k+m) m \overline z^{k+m-1} \, z^{m-1} + a_1 (k+m-1)(m-1) \overline z^{k+m-2}\, z^{m-2} + \dots + a_{m-1} (k+1) \overline z^k
$$
as well as
\vskip 0.2 cm
$\overline z \, \frac{\partial}{\partial \overline z}\, u_{k,m}(z,\overline z)$
$$= (k+m) \overline z^{k+m}\, z^m + a_1(k+m-1) \overline z^{k+m-1} \, z^{m-1} + \dots + a_{m-1} (k+1) \overline z^{k+1}\, z + a_m k \overline z^k,$$
which implies that the function $u_{k,m}$ is an  eigenfunction for the eigenvalue $\mu=k+m$ and from \eqref{spect1} we obtain, comparing  the highest exponents of $\overline z$ and $z,$ 
$$(k+m)m-a_1(k+m-1) = -(k+m)a_1,$$
hence $a_1=-(k+m)m.$ Comparing the next lower exponents we get
$$a_1(k+m-1)(m-1)-a_2(k+m-2)= -a_2(k+m)$$ 
and $a_2=\frac{1}{2}\, (k+m)(k+m-1)m(m-1).$
In general we find that for $j=1,2, \dots m$
$$a_j=\frac{(-1)^j (k+m)!\,  m!}{j! \, (k+m-j)!\, (m-j)!},$$
which proves \eqref{eig1}.

In order to show \eqref{eig2} we set 
$$v_{k,m}(z,\overline z)= \overline z^{k} \, z^{k+m} + b_1 \overline z^{k-1}\, z^{k+m-1} + b_2 \overline z^{k-2}\, z^{k+m-2} + \dots + b_{k-1} \overline z \, z^{m+1} + b_k  z^m
$$
and compute
\vskip 0.2 cm
$\frac{\partial^2}{\partial z \partial \overline z}\, v_{k,m}(z,\overline z)$
$$ = k(k+m)  \overline z^{k-1} \, z^{k+m-1} + b_1(k-1) (k+m-1) \overline z^{k-2}\, z^{k+m-2} + \dots + b_{k-1}  \overline z \,z^{m+1}
$$
as well as
\vskip 0.2 cm
$$\overline z \, \frac{\partial}{\partial \overline z}\, v_{k,m}(z,\overline z)
= k \overline z^{k}\, z^{k+m} +b_1 (k-1)\overline z^{k-1}\, z^{k+m-1} +  \dots + b_{k-1}  \overline z\, z^{m+1} $$
which implies that the function $v_{k,m}$ is an  eigenfunction for the eigenvalue $\mu=k,$ for each $m\in \mathbb{N}$ and from \eqref{spect1} we obtain, comparing  the highest exponents of $\overline z$ and $z,$ 
$$k(k+m)-b_1(k-1) = -k b_1,$$
hence $b_1=-(k+m)k.$ Comparing the next lower exponents we get
$$b_1(k-1)(k+m-1)-b_2(k-2)= -b_2k$$ 
and $b_2=\frac{1}{2}\, (k+m)(k+m-1)k(k-1).$
In general we find that for $j=1,2, \dots k$
$$b_j=\frac{(-1)^j (k+m)!\,  k!}{j! \, (k+m-j)!\, (k-j)!},$$
which proves \eqref{eig2}.

\end{proof}

Now we are able to prove
\begin{theorem}\label{spec1} Let $n=1$ and $\varphi (z)=|z|^2.$ The spectrum of $\square_{\varphi, 0}$ consists of all non-negative integers $\{ 0, 1, 2, \dots\}$ each of which is of infinite multiplicity, so $0$ is the bottom of the essential spectrum.  The spectrum of $\square_{\varphi, 1}$ consists of all positive integers $\{ 1, 2, 3, \dots\}$  each of which is of infinite multiplicity.

\end{theorem}
\begin{proof} We already know that the whole Bergman space $A^2(\mathbb{C}, e^{-|z|^2})$ is contained in the eigenspace of the eigenvalue $0$ of the operator  $\square_{\varphi, 0}$ 
and, for any positive integer $k,$ the antiholomorphic monomial $\overline z^k$ is an eigenfunction for the eigenvalue $\mu =k.$
In addition all functions of the form $\overline z^\nu \, z^\kappa$ with $\nu, \kappa \in \mathbb{N}_0$ can be expressed as a linear combination of functions of the form \eqref{eig1} and \eqref{eig2}. For a fixed $k\in \mathbb{N}$ the functions of the \eqref{eig2} have infinite multiplicity as the parameter $m\in \mathbb{N}_0$ is free. So all eigenvalues are of infinite multiplicity. All the eigenfunctions considered so far yield a complete orthogonal basis of $L^2(\mathbb{C}, e^{-|z|^2}),$ since the Hermite polynomials $\{H_0(x)H_k(y), H_1(x)H_{k-1}(y), \dots, H_k(x) H_0(y)\}$ for $k=0,1,2, \dots$ form a complete orthogonal system in $L^2(\mathbb{R}^2, e^{-x^2-y^2})$ (see for instance \cite{Fo3}) and since $x=1/2(z+\overline z) \ , \ y=i/2 (\overline z -z)$ we can apply the Lemma \ref{davies}  and obtain that the spectrum of  $\square_{\varphi, 0}$ is $\mathbb{N}_0.$

The statement for the spectrum of $\square_{\varphi, 1}$ follows from \eqref{boxn}.
\end{proof}
For several variables we can adopt the method from above to obtain the following result
\begin{theorem}\label{sqspec} Let $n>1$ and $\varphi (z)= |z_1|^2 + \dots + |z_n|^2$ and $0\le q \le n.$ 
 The spectrum of $\square_{\varphi, q}$ consists of all  integers $\{ q, q+1, q+2, \dots\}$ each of which is of infinite multiplicity. 
 \end{theorem}
 \begin{proof} 
Recall that the $\ovprt$-Neumann Laplacian $\square_{\varphi,q}$ acts diagonally and that
$$\square_{\varphi,q} \, u = \sum_{|J|=q}\kern-1pt{}^{\prime} \left (
 -\frac{1}{4}\, \triangle u_J + \sum_{j=1}^n \overline z_j \, u_{J \overline z_j} +q \, u_J \right )\,d\overline z_J .$$
 The factor $q$ in the last formula is responsible for the fact that the eigenvalues start with $q,$ which can be seen, in each component separately, by
 $$-\frac{1}{4}\, \triangle u_J + \sum_{j=1}^n \overline z_j \, u_{J \overline z_j} = (\mu -q)\, u_J .$$
 Now let 
  $k_1,k_2,\dots , k_n \in \mathbb{N}_0$ and $m_1, m_2, \dots ,m_n \in \mathbb{N}.$ Then  the function
 $$u_{k_1,m_1}(z_1,\overline z_1)\, u_{k_2,m_2}(z_2,\overline z_2)\dots \, u_{k_n,m_n}(z_n,\overline z_n)$$ 
 is component of an eigenfunction for the eigenvalue $\sum_{j=1}^n (k_j+m_j)$ of the operator $\square_{\varphi,q},$ which follows from \eqref{boxq} and \eqref{eig1}.
 
 Similarly it follows from \eqref{boxq}  and \eqref{eig2} that for $k,_1,k_2,\dots , k_n \in \mathbb{N}$ and $m_1, m_2, \dots ,m_n \in \mathbb{N}_0$ the function 
$$ v_{k_1,m_1}(z_1,\overline z_1)\, v_{k_2,m_2}(z_2,\overline z_2)\dots \, v_{k_n,m_n}(z_n,\overline z_n)$$ 
is an eigenfunction for the eigenvalue $\sum_{j=1}^n k_j.$

All other possible $n$-fold products with factors $u_{k_j, m_j}$ or  $v_{k_j, m_j}$ (also mixed) appear as eigenfunctions of $\square_{\varphi,q}.$

From this we obtain that all expressions of the form $z_1^{\alpha_1} \, \overline z_1^{\beta_1} \,  \dots \,z_n^{\alpha_n} \, \overline z_n^{\beta_n}$ for arbitrary $\alpha_j ,\beta_j  \in \mathbb{N}_0,\ j=1,\dots ,n,$ can be written as   a linear combination of eigenfunctions of $\square_{\varphi, q},$ which proves that all these eigenfunctions constitute a complete basis in $L^2_{(0,q)}(\mathbb{C}^n, e^{-|z|^2}),$ see the proof of Theorem \ref{spec1}. So we can again apply Lemma \ref{davies}.

 \end{proof}
\begin{rem} 
(i) Since in all cases the essential spectrum is non-empty, the corresponding $\ovprt$-Neumann operator fails to be with compact resolvent (see for instance \cite{Dav}).

(ii) If one considers the weight function 
$$\varphi (z)= (|z_1|^2+ |z_2|^2 + \dots + |z_n|^2)^{\alpha} \ \ {\text{for}} \  \alpha >1$$
 the situation is completely different: the operators $\square_{\varphi, q}$ are with compact resolvent (see \cite{HaHe}), so the essential spectrum must be empty.

\end{rem}
\vskip 0.3 cm

We can use the results from above to settle the corresponding questions for the so-called Witten-Laplacian which is defined on $L^2(\mathbb C^n).$

 For this purpose we set $Z_k=\frac{\partial}{\partial \overline z_k}+ \frac{1}{2}\, \frac{\partial \varphi}{\partial \overline z_k}$ and $Z^*_k=-\frac{\partial}{\partial z_k}+ \frac{1}{2}\, \frac{\partial \varphi}{\partial z_k}$ and we consider $(0,q)$-forms $h=\sum_{|J|=q} \kern-1pt{}^{\prime}  \, h_J \, d\overline z_J,$ where $\sum \kern-1pt{}^{\prime} $ means that we sum up only increasing multiindices $J=(j_1,\dots ,j_q)$ and where $d\overline z_J = d\overline z_{j_1} \wedge \dots \wedge d\overline z_{j_q}.$ We define 
\begin{equation}
\overline D_{q+1} h = \sum_{k=1}^n \sum_{|J|=q} \kern-1pt{}^{\prime}  \, Z_k(h_J)\, d\overline z_k \wedge d\overline z_J
\end{equation}
and
\begin{equation}
\overline D_q ^* h = \sum_{k=1}^n \sum_{|J|=q} \kern-1pt{}^{\prime}  \, Z_k^*(h_J)\, d\overline z_k  \rfloor d\overline z_J ,
\end{equation}
where $d\overline z_k  \rfloor d\overline z_J$ denotes the contraction, or interior multiplication\index{interior multiplication} by $d\overline z_k,$ i.e. we have
$$\langle \alpha , d\overline z_k  \rfloor d\overline z_J \rangle = \langle d\overline z_k \wedge \alpha, d\overline z_J \rangle$$
for each $(0,q-1)$-form $\alpha .$

The complex Witten-Laplacian on $(0,q)$-forms is then given by 
\begin{equation}
\Delta^{(0,q)}_{\varphi} = \overline D_q \, \overline D_q^* + \overline D_{q+1}^* \, \overline D_{q+1},
\end{equation}
\index{$\Delta^{(0,q)}_{\varphi}$}
for $q=1,\dots, n-1.$

The general $\ovli D$-complex  has the form
\begin{equation}\label{schr14} 
L_{(0,q-1)}^2(\mathbb C^n ) \underset{\underset{\ovli D_q^* }{\longleftarrow }}{\overset{\ovli D_q }{\longrightarrow}}
L^2_{(0,q)}(\mathbb C^n )   \underset{\underset{\ovli D_{q+1}^* }{\longleftarrow }}{\overset{\ovli D_{q+1} }{\longrightarrow}}
L^2_{(0,q+1)}(\mathbb C^n )  \;. 
\end{equation}

It follows that
\begin{equation}
\overline D_{q+1} \, \Delta^{(0,q)}_{\varphi } = \Delta^{(0,q+1)}_{\varphi } \, \overline D_{q+1} \ \  {\text{and}} \ \ 
\overline D_{q+1}^* \, \Delta^{(0,q+1)}_{\varphi } = \Delta^{(0,q)}_{\varphi } \, \overline D_{q+1}^* .
\end{equation}
We remark that 
\begin{equation}
\overline D_q ^* h = \sum_{k=1}^n \sum_{|J|=q} \kern-1pt{}^{\prime}  \, Z_k^*(h_J)\, d\overline z_k  \rfloor d\overline z_J  =  \sum_{|K|=q-1} \kern-4pt{}^{\prime}  \sum_{k=1}^n  \, Z_k^*(h_{kK})\, d\overline z_K.
\end{equation}
In particular we get for a function $v\in L^2(\mathbb{C}^n)$
\begin{equation}
\Delta^{(0,0)}_{\varphi }v =  \overline D_{1}^* \, \overline D_{1} v = \sum_{j=1}^n Z_j^* Z_j(v),
\end{equation}
and for a $(0,1)$-form $g=\sum_{\ell=1}^n g_\ell\, d\ovli z_\ell \in L^2_{(0,1)}(\mathbb C^n ) $ we obtain
\begin{equation}
\Delta^{(0,1)}_{\varphi } g= (\overline D_1 \, \overline D_1^* + \overline D_{2}^* \, \overline D_{2}) g =
(\Delta^{(0,0)}_{\varphi } \otimes I)g  +  M_{\varphi} g,
\end{equation}
where we set
$$
M_{\varphi } g = \sum_{j=1}^n \left (\sum_{k=1}^n  \frac{\partial^2 \varphi}{\partial z_k \partial
    \ovli z_j} \, g_k \right )\, d\ovli z_j
$$
and 
\begin{equation*}
\begin{array}{ll}
(\Delta^{(0,0)}_{\varphi } \otimes I )\, g = \sum_{k=1}^n \Delta^{(0,0)}_\varphi \, g_k\, d\ovli z_k.
\end{array}
\end{equation*}

In general we have that
\begin{equation}\label{witten}
\Delta^{(0,q)}_{\varphi } = e^{-\varphi/2} \,  \square_{\varphi, q} \, e^{\varphi/2},
\end{equation}
for $q=0,1\dots, n.$

For more details see  \cite{HaHe} and \cite{Ga}.

In our case $\varphi (z)= |z_1|^2 + \dots + |z_n|^2$ we get
\begin{equation}\label{witten1}
\Delta^{(0,q)}_{\varphi } h =
\sum_{|J|=q}\kern-1pt{}^{\prime} \left (
 -\frac{1}{4}\, \triangle h_J + \frac{1}{2} \, \sum_{j=1}^n (\overline z_j \, h_{J \overline z_j}-z_j\, h_{J z_j})
 + \frac{1}{4} \, |z|^2 \, h_J +(q-\frac{n}{2})\, h_J
   \right )\,d\overline z_J ,
\end{equation}
for
$$h=\sum_{|J|=q}\kern-1pt{}^{\prime} h_J\,d\overline z_J  \in {\text{dom}}\, \Delta^{(0,q)}_{\varphi }  \subseteq L^2_{(0,q)}(\mathbb{C}^n).$$ 
The spectrum of $\Delta^{(0,0)}_{\varphi } ,$ even in a more general form, was calculated  by X. Ma and G. Marinescu, \cite{Ma1}, \cite{Ma}.

Using \eqref{witten} and Lemma \ref{davies} we get that $\Delta^{(0,q)}_{\varphi }$  and $\square_{\varphi, q}$ have the same spectrum. Hence by Theorem  \ref{sqspec} we obtain
\begin{theorem}Let $\varphi (z)= |z_1|^2 + \dots + |z_n|^2$ and $0\le q \le n.$ 
 The spectrum of the Witten-Laplacian $\Delta^{(0,q)}_{\varphi }$ consists of all  integers $\{ q, q+1, q+2, \dots\}$ each of which is of infinite multiplicity.
 \end{theorem}

\bibliographystyle{amsplain}
\bibliography{mybibliography}

\end{document}